\documentclass{amsart}

\usepackage{amsfonts}
\usepackage{amsmath}
\usepackage{amsthm}
\usepackage{graphicx}
\usepackage{wasysym}
\usepackage{amssymb}
\usepackage{mathrsfs}
\usepackage{amscd}
\usepackage{lscape}
\usepackage{longtable}
\usepackage{rotating}
\usepackage{phonetic}
\usepackage{etex}
\usepackage{float}
\usepackage{pgf,tikz}
\usetikzlibrary{arrows}

\input prepictex
\input pictex
\input postpictex

\newtheorem{lemma}{Lemma}[section]
\newtheorem{thm}[lemma]{Theorem}

\newtheorem{prop}[lemma]{Proposition}
\newtheorem{corol}[lemma]{Corollary}

\newtheorem{example}[lemma]{Example}

\title{On the poset of vector partitions}
\author{NATALIE AISBETT}
\address{School of Mathematics and Statistics\\
University of Sydney, NSW, 2006\\
Australia}
\email{N.Aisbett@maths.usyd.edu.au}
\date{}

\begin{document}
\maketitle{}
\numberwithin{figure}{section}
\numberwithin{equation}{section}

\begin{abstract}
We consider the poset of vector partitions of $[n]$ into $s$ components, denoted $\Pi_{n,s}$, which was first defined by Stanley in \cite{st}. Sagan has shown in \cite{sa} that this poset is CL-shellable, and hence has the homotopy type of a wedge of spheres of dimension $(n-2)$. We extend on this result to show that $\Pi_{n,s}$ is edge-lexicographic shellable. We then use this edge-labeling to find a recursive expression for the number of spheres, and show that when $s=1$ the number of spheres is equal to the number of complete non-ambiguous trees, first defined in \cite{abb}. \\
\end{abstract}

\section{Introduction }

For any positive integer $n$, we let $[n]$ denote the set of integers $\{1,2,\ldots,n\}$. A \emph{vector partition} of $[n]$ into $s+1$ components, or an \emph{$(s+1)$-partition of $[n]$}, is an $(s+1)$-tuple $(P,w^1,\ldots,w^s)$ where $P$ is a partition of $[n]$, and each $w^i$ is a labeling of the parts of $P$ so that a part with $\alpha$ elements is labeled by a subset of cardinality $\alpha$, and the labels form a partition of $[n]$. We define a poset $\Pi_{n,s+1}$ with elements the $(s+1)$-partitions of $[n]$ union a least element $\hat{0}$ with relation $(P,w^1,\ldots,w^s) \le (P',w^{1\prime},\ldots, w^{s\prime})$ if every part in $P'$ is the union of parts in $P$, and the corresponding label sets in $w^{i\prime}$ are the union of the label sets in $w^i$ for each $i$. The poset of \emph{vector partitions of $[n]$} into $s+1$ components, or \emph{$(s+1)$-partitions of $[n]$}, denoted $\overline{\Pi}_{n,s+1}$ is the induced subposet of $\Pi_{n,s+1}$ with the least and greatest elements removed, i.e. it is the induced subposet with elements 
$$\overline{\Pi}_{n,s+1}=\Pi_{n,s+1}-\{\hat{0}, ([n],\ldots, [n])\}.$$   
The relation $P \le P'$ where the parts of $P'$ are the union of the parts in $P$ forms the well known poset of partitions of $[n]$, denoted by $\Pi_n$.\\

The poset of vector partitions is an example of an exponential structure. Exponential structures were defined by Stanley in \cite{st}, where this particular type of exponential structure is mentioned as Example 2.5. In \cite{sa}, Sagan studies the poset of vector partitions and shows that it admits a recursive atom ordering, which is equivalent to being CL-shellable (see \cite{bw,bw3}). \\

The poset $\Pi_{n,2}$ is a member of a family of posets that includes the poset of labeled subforests defined in \cite{br}. We can see this by taking the partition $P$ in an element $(P,w^1) \in \Pi_{n,2}$ to be a partition of the vertices of $K_n$. In \cite{br}, Babson and Reiner define a poset in a similar manner to $\Pi_{n,2}$, however their underlying graph is any tree with $n$ vertices rather than $K_n$. When the underlying graph is $\mathrm{Path}_n$ the order complex of this poset is the dual simplicial complex to the permutohedron.\\

\textbf{Acknowledgements}
I would like to thank Russ Woodroofe for his helpful suggestions on how to improve this paper. As suggested by Woodroofe, I will be considering the method used in \cite{rw} for giving an edge labeling of a dual poset, which can possibly be used to find another expression for the number of spheres in $\Delta(\Pi_{n,s})$.

\section{Definitions}

This section contains a summary of the terminology used in this paper. We discuss simplicial complexes, poset theory, order complexes, and shelling orders. The reader is advised to see \cite[Wachs, Poset Topology]{gc} for a complete introduction to this theory.\\

A \emph{simplicial complex} $\Delta$ with finite vertex set $S$, is a set of subsets of $S$ such that
\begin{itemize}
\item $\{i\} \in \Delta$ for every $i \in S$, and 
\item if $I \in \Delta$, and $J \subseteq I$, then $J \in \Delta$.
\end{itemize} Elements of a simplicial complex are called \emph{faces}, and the dimension of a face $F$ is equal to $|F|-1$. Faces of dimension 0 are called \emph{vertices}, faces of dimension 1 are called \emph{edges}, and a \emph{facet} is a face that is not properly contained in any other face. The dimension of a simplicial complex is the maximal dimension of all its faces. A simplicial complex is \emph{pure} if all its facets are the same dimension. \\

A \emph{poset} is a finite set $P$ and a relation $\le$ on $P$, that satisfies: 
\begin{itemize}
\item $x \le x$ for all $x \in P$, 
\item if $x \le y$ and $y \le x$, then $x=y$,
\item if $x \le y$ and $y \le z$, then $x \le z$.
\end{itemize}
A \emph{length $m$ chain} in $P$ is a totally ordered subset $\{x_0<x_1<\cdots<x_m\}$ of $P$. A \emph{maximal chain} is a chain that is maximal with respect to inclusion. A poset $S$ is an induced subposet of a poset $P$ if $S\subseteq P$, and $x \le y$ in $S$ if and only if $x \le y$ in $P$. For any $x \le y$ in $P$, the \emph{interval} $[x,y]$ is the induced subposet of $P$ with elements all $z \in P$ such that $x \le z \le y$. The \emph{length} of an interval is the maximum of the length of all maximal chains in the interval. For a given subset $S$ of $P$, an element $x \in S$ is the \emph{greatest} (respectively \emph{least}) element of $S$ if $s \le x$ (respectively  $x \le s$) for all $s \in S$. The greatest element in a poset is denoted $\hat{1}$, and the least element is denoted $\hat{0}$. A poset if \emph{bounded} if it contains both $\hat{0}$ and $\hat{1}$. A poset $P$ is \emph{pure} if all maximal chains in $P$ have the same length. A poset is \emph{graded} if it is bounded and pure. If $P$ is a bounded poset, the \emph{proper part} of $P$, denoted by $\overline{P}$, is the induced subposet with elements $P -\{\hat{0},\hat{1}\}$. If $x$ and $y$ are elements of a poset $P$, the notation $x \lessdot y$ means that $x <y$, and there is no $z \in P$ such that $x<z<y$. If $x \lessdot y$, then we say that $x < y$ is an \emph
{edge} of $P$. An \emph{atom} of a poset $P$ with least element $\hat{0}$ is an element $a \in P$ such that $\hat{0} \lessdot a$ is an edge of $P$. Posets $P$ and $Q$ are \emph{isomorphic}, denoted $P \cong Q$, if there is a bijective map $\phi:P \rightarrow Q$ such that $\phi(x) \le \phi(y)$ in $Q$ if and only if $x \le y$ in $P$.\\

The \emph{order complex} of a poset $P$, denoted $\Delta(P)$, is the simplicial complex with vertex set the elements of $P$, and with faces given by the chains of $P$. \\

It is clear from the definition that the poset $\Pi_{n,s}$ is graded. In this paper, we will be studying topology of the order complex of the poset $\overline{\Pi}_{n,s}$. To do this we use a tool called shellability which is defined in the next paragraph.\\

If $F$ is a face of a simplicial complex $\Delta$, we denote by $\langle F \rangle$ the set of faces $G \in \Delta$ such that $G \subseteq F$. A simplicial complex is \emph{shellable} if there is a linear ordering $F_1,F_2,\ldots, F_k$ of its facets so that for all $i \in \{2,\ldots,k\}$, the set $\langle F_i \rangle \cap (\cup_{j=1}^{i-1}\langle F_j\rangle)$ is a pure $(|F_i|-2)$-dimensional simplicial complex. \\

\begin{thm}
[Bj\"{o}rner and Wachs \cite{bw}] A shellable simplicial complex has the homotopy type of a wedge of spheres, where for each $m$, then number of $m$-dimensional spheres is the number of $m$-dimensional facets $F_i$ whose boundary is contained in $\cup_{j=1}^{i-1}\langle F_j \rangle$. Such facets are called \emph{homology facets}. \label{asdfa}
\end{thm} 

There are well known techniques which can be used to show that the order complex of a graded poset is shellable. The topology of the order complex of a graded poset is not interesting since it is a  contractible space. However, a shelling of the order complex of a bounded poset induces a shelling on the order complex of the proper part of this poset, as described in Proposition \ref{htgdyh} below. In this paper we use a technique called EL-shellability, which we introduce in this section, to show that $\Delta(\Pi_{n,s})$ and hence $\Delta(\overline{\Pi}_{n,s})$ is shellable.

\begin{prop}
Suppose that $P$ is a bounded poset, and that $C_1,C_2,\ldots,C_t$ is an ordering of the maximal chains of $P$ that gives a shelling of the order complex $\Delta(P)$. Then $\overline{C}_1,\overline{C}_2,\ldots,\overline{C}_t$ is an ordering of the maximal chains of $\bar{P}$ that gives a shelling of $\Delta(\overline{P})$. Here $\overline{C}_i$ is the totally ordered subset $C_i - \{\hat{0},\hat{1}\}$. \label{htgdyh}
\end{prop}

We will now describe EL-shellability, a method developed by  Bj\"{o}rner and Wachs in \cite{bj,bw2,bw,bw4,bw3} to show that the order complex of a bounded poset is shellable. If $P$ is a bounded poset, an \emph{edge labeling} of $P$ is a map $\lambda: \mathcal{C}(P) \rightarrow \Lambda$, where $\mathcal{C}(P)$ is the set of edges of $P$, and $\Lambda$ is a totally ordered set. If $P$ has an edge labeling $\lambda$, and $C= x_1 \lessdot \cdots\lessdot x_m$ is a chain of $P$, then we call the word $\lambda(C)=\lambda(x_1\lessdot x_2)\lambda(x_2\lessdot x_3)\ldots\lambda(x_{m-1}\lessdot x_m)$ the \emph{chain label} of the chain $C$. We say that a chain $C$ is \emph{increasing} if $\lambda(C)$ is strictly increasing, and that $C$ is \emph{decreasing} if $\lambda(C)$ weakly decreasing.\\

Suppose that $P$ is a bounded poset. An \emph{edge-lexicographical labeling} or \emph{EL-labeling} of $P$ is an edge labeling such that for each interval $[x,y]$ in $P$, there is a unique increasing maximal chain, which lexicographically precedes the other maximal chains of $[x,y]$. A bounded poset that admits an EL-labeling is said to be \emph{edge-lexicographic shellable} or \emph{EL-shellable}. \\

\begin{thm}[Bj\"{o}rner \cite{bj}, Bj\"{o}rner and Wachs \cite{bw}]
Suppose $P$ is a bounded poset with an edge lexicographical labeling. Then the lexicographical ordering of the maximal chains of $P$ is a shelling of $\Delta(P)$. The induced order on the maximal chains of $\overline{P}$ is a shelling of $\Delta(\overline{P})$, where the number of $i$-spheres in $\Delta(\overline{P})$ is equal to the number of decreasing maximal chains of length $i+2$ in $P$.\label{kjvgm}
\end{thm}

\section{An $EL$-labeling of $\Pi_{n,s+1}$}

Suppose $(P,w^1,\ldots, w^s)$ is a vertex of $\Pi_{n,s+1}-\{\hat{0}\}$. We call $w^i$ the $i$th labeling of $P$, and if $I$ is a part in $P$, we call its corresponding label in $w^i$ the $i$th label of $I$. From now on, we will list the parts of $P$ in the order of their minimal element (a part $I$ is listed earlier than a part $J$ if the lowest integer in $I$ is lower than the lowest integer in $J$), and we will represent each $w^i$ by ordering the labelings by the parts they label in $P$. For example $$(P,w^1,w^2) = (\{1,4,6\}\{2,3\}\{5\},\{2,3,5\}\{4,6\}\{1\},\{3,5,6\}\{2,4\}\{1\})$$ means that the part $\{1,4,6\}$ has 1st label $\{2,3,5\}$ and 2nd label $\{3,5,6\}$, the part $\{2,3\}$ has 1st label $\{4,6\}$ and 2nd label $\{2,4\}$, and the part $\{5\}$ has $1$st label $\{1\}$ and $2$nd label $\{1\}$. 

\begin{example}
We depict the Hasse diagram of the poset $\overline{\Pi}_{3,2}$. 

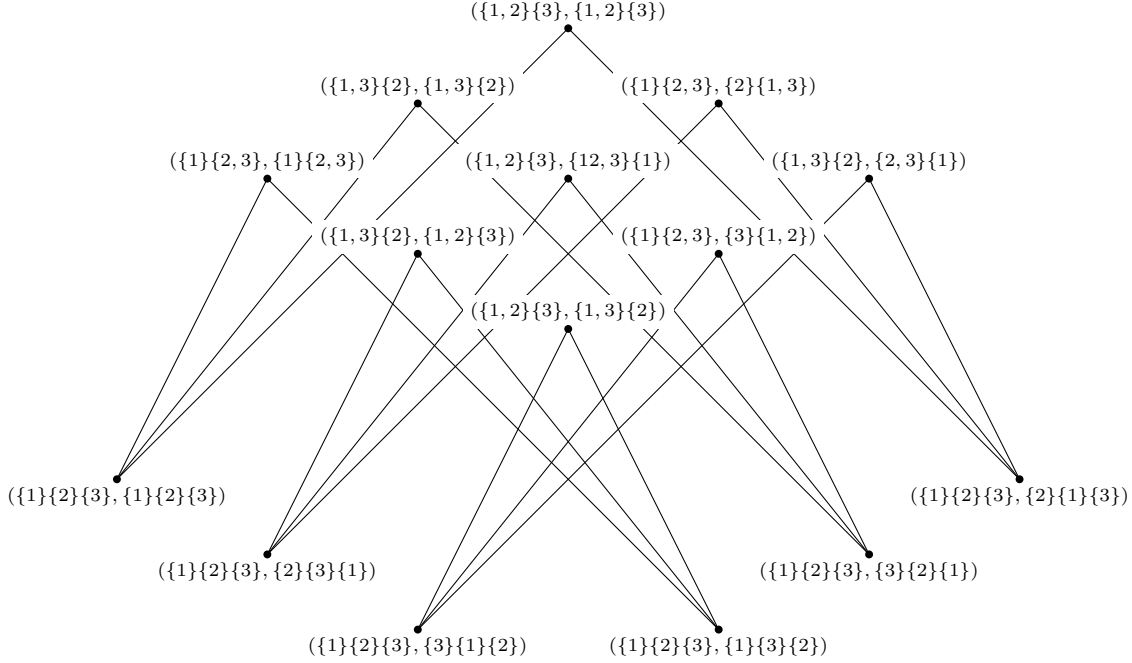
\begin{center}
\begin{tikzpicture}[line cap=round,line join=round,>=triangle 45,x=1.0cm,y=1.0cm]
\draw (-2,-2)-- (0,2);
\draw (-2,-2)-- (2,3);
\draw (-2,-2)-- (4,4);
\draw (-4,-1)-- (-2,3);
\draw (-4,-1)-- (0,4);
\draw (-4,-1)-- (2,5);
\draw (-6,0)-- (-4,4);
\draw (-6,0)-- (-2,5);
\draw (-6,0)-- (0,6);
\draw (2,-2)-- (0,2);
\draw (2,-2)-- (-2,3);
\draw (2,-2)-- (-4,4);
\draw (4,-1)-- (2,3);
\draw (4,-1)-- (0,4);
\draw (4,-1)-- (-2,5);
\draw (6,0)-- (4,4);
\draw (6,0)-- (2,5);
\draw (6,0)-- (0,6);
\begin{scriptsize}

\draw (-2,3) node[above,fill=white] {$(\{1,3\}\{2\},\{1,2\}\{3\})$};
\draw (-4,4) node[above,fill=white] {$(\{1\}\{2,3\},\{1\}\{2,3\})$};
\draw (2,3) node[above,fill=white] {$(\{1\}\{2,3\},\{3\}\{1,2\})$};
\draw (0,4) node[above,fill=white] {$(\{1,2\}\{3\},\{12,3\}\{1\})$};
\draw (-2,5) node[above,fill=white] {$(\{1,3\}\{2\},\{1,3\}\{2\})$};
\draw (4,4) node[above,fill=white] {$(\{1,3\}\{2\},\{2,3\}\{1\})$};
\draw (2,5) node[above,fill=white] {$(\{1\}\{2,3\},\{2\}\{1,3\})$};
\draw (0,6) node[above,fill=white] {$(\{1,2\}\{3\},\{1,2\}\{3\})$};
\draw (0,2) node[above,fill=white] {$(\{1,2\}\{3\},\{1,3\}\{2\})$};
\draw (-2,-2) node[below,fill=white] {$(\{1\}\{2\}\{3\},\{3\}\{1\}\{2\})$};
\draw (-4,-1) node[below,fill=white] {$(\{1\}\{2\}\{3\},\{2\}\{3\}\{1\})$};
\draw (-6,0) node[below,fill=white] {$(\{1\}\{2\}\{3\},\{1\}\{2\}\{3\})$};
\draw (2,-2) node[below,fill=white] {$(\{1\}\{2\}\{3\},\{1\}\{3\}\{2\})$};
\draw (6,0) node[below,fill=white] {$(\{1\}\{2\}\{3\},\{2\}\{1\}\{3\})$};
\draw (4,-1) node[below,fill=white] {$(\{1\}\{2\}\{3\},\{3\}\{2\}\{1\})$};

\fill  (-2,3) circle (1.5pt);
\fill  (-4,4) circle (1.5pt);
\fill  (2,3) circle (1.5pt);
\fill  (0,4) circle (1.5pt);
\fill  (-2,5) circle (1.5pt);
\fill  (4,4) circle (1.5pt);
\fill  (2,5) circle (1.5pt);
\fill  (0,6) circle (1.5pt);
\fill  (-2,-2) circle (1.5pt);
\fill  (-4,-1) circle (1.5pt);
\fill  (-6,0) circle (1.5pt);
\fill  (2,-2) circle (1.5pt);
\fill  (6,0) circle (1.5pt);
\fill  (0,2) circle (1.5pt);
\fill  (4,-1) circle (1.5pt);
\end{scriptsize}
\end{tikzpicture}
\end{center}

\end{example}

Recall that $\Pi_n$ denotes the poset of partitions of $[n]$, whose elements are partitions of $[n]$, with relation $x \le y$ if and only if every part in $y$ is the union of parts in $x$. We will now describe an EL-labeling $\chi$ of the poset $\Pi_n$, which is defined by Wachs in \cite{wa}. 

\begin{thm}[Wachs \cite{wa}]
\label{jhgchgc} Suppose $x \lessdot y$ in $\Pi_n$, and that $y$ is obtained from $x$ by merging two parts $B_1$ and $B_2$. Let $\chi:\mathcal{C}(\Pi_n) \rightarrow \mathbb{Z}$ be defined as $$\chi(x \lessdot y) = \max(B_1 \cup B_2).$$ The map $\chi$ defines an EL-labeling of the poset $\Pi_n$. \label{wathm} 
\end{thm}

\begin{prop}
Any interval $[(P,w^1,\ldots, w^s),(P'',w^{1\prime \prime},\ldots,w^{s\prime \prime})]$ in $\Pi_{n,s+1}$ (where $(P,w^1,\ldots, w^s) \ne \hat{0}$) is isomorphic to the interval $[P,P'']$ in the poset $\Pi_n$. \label{hgchc}
\end{prop}

\begin{proof}
It is clear by the definition of $\Pi_{n,s+1}$, that the map $$(P',{w}^{1\prime},\ldots, w^{s\prime}) \mapsto P'$$ defines an isomorphism of posets $$[(P,w^1,\ldots, w^s),(P'',w^{1\prime\prime},\ldots,w^{s\prime\prime})] \cong [P,P''].$$
\end{proof}

Henceforth, we will identify atoms $$(\{1\}\{2\}\ldots\{n\},\{w^1_1\}\ldots\{w^1_n\},\ldots\ldots, \{w^s_1\}\ldots\{w^s_n\})$$ in $\Pi_{n,s+1}$ with the word 
$$w_1^1\ldots w_n^1\ldots\ldots w^s_1\ldots w^s_n.$$ With this identification, the atoms can be ordered by the lexicographical ordering on $[n]^{ns}$. For each element $(P,w^1,\ldots,w^s) \in \Pi_{n,s+1} -\{\hat{0}\}$, there is a unique atom $w_1^1\ldots w_n^1\ldots\ldots w^s_1\ldots w^s_n$ of $\Pi_{n,s+1}$ which is less than $(P,w^1,\ldots,w^s)$, and is earliest in the lexicographic order on the atoms. We call $w_1^1\ldots w_n^1 \ldots\ldots w^s_1 \ldots w^s_n$ the \emph{atom of $(P,w^1,\ldots,w^s)$}, and denote it by $A(P,w^1,\ldots,w^s)$. If $(P,w^1,\ldots,w^s) \in \Pi_{n,s+1}$, we denote by $w_j^i$ the $(n(i-1)+j)$th element of $A(P,w^1,\ldots, w^s)$. The atom of $(P,w^1,\ldots,w^s)$ can be obtained as follows: Suppose that $\{k_1,\ldots,k_p\}$ is a part in $P$ with $i$th label $\{j_1,\ldots,j_p\}$, where $k_1< \cdots<k_p$ and $j_1 < \cdots < j_p$. Then $w^i_{k_\alpha} = j_{\alpha}$ for all $\alpha \in \{1,\ldots,p\}$. For example $$(\{1\}\{2\}\{3\}\{4\}\{5\}\{6\}\{7\}\{8\},\{2\}\{1\}\{5\}\{3\}\{4\}\{8\}\{6\}\{7\})$$ is the atom of $$(\{1,4,8\}\{2,3,7\}\{5,6\},\{2,3,7\}\{1,5,6\}\{4,8\})$$ and it is identified with the word $$21534867.$$ \\

We will now define an edge labeling $\lambda: \mathcal{C}(\Pi_{n,s+1}) \rightarrow \mathbb{Z} \times \mathbb{Z} \times \mathbb{Z}$ of $\Pi_{n,s+1}$, where the ordering on $\mathbb{Z} \times \mathbb{Z} \times \mathbb{Z}$ is the lexicographical ordering. This labeling is similar to the labelings defined in \cite{wa} and \cite{rw}, in that it labels the lowest edges with ``some version'' of 0, and labels some other edges with ``some version'' of a labeling of the partition lattice, in this case Wachs' label in Theorem \ref{wathm}:
\begin{itemize}
\item Suppose $$(P,w^1,\ldots,w^s) \lessdot (P',w^{1\prime},\ldots, w^{s \prime}),$$ and $$A(P,w^1,\ldots ,w^s) \ne A(P',w^{1\prime},\ldots, w^{s\prime}).$$ Then $$\lambda((P,w^1,\ldots,w^s) \lessdot (P',w^{1\prime},\ldots, w^{s\prime})) = (k,i,j)$$ where $(k,i,j)$ is the lexicographically first element of $\mathbb{Z} \times \mathbb{Z} \times \mathbb{Z}$ such that $w_k^i \ne w_k^{i\prime}$ and $w_k^{i\prime}=j$.\\ 
\item Suppose $$(P,w^1,\ldots ,w^s) \lessdot (P',w^{1\prime},\ldots, w^{s\prime}),$$ and $$A(P,w^1,\ldots ,w^s) = A(P',w^{1\prime},\ldots, w^{s \prime}).$$ Then $$\lambda((P,w^1,\ldots,w^s) \lessdot (P',w^{1\prime},\ldots, w^{s\prime})) = (n,\max{(I \cup J)},0),$$ where $I$ and $J$ are the parts in $P$ that are merged to give $P'$.\\
\item For any edge $$\hat{0} \lessdot a,$$ where $a$ is an atom of $\Pi_{n,s+1}$, $$\lambda(\hat{0} \lessdot a)=(n-1,s+m,0)$$ where $a$ is the $m$th atom in the lexicographical ordering on the atoms of $\Pi_{n,s+1}$.
\end{itemize}

We make the following observations about this edge labeling before proceeding with the proof of Theorem \ref{kjgkjbkb}, which shows that $\lambda$ is an EL-labeling.\\
\begin{lemma}
\label{htchft}
The edge labeling $\lambda$ satisfies the following five conditions:
\begin{itemize}
\item[(1)] If $(P,w^1,\ldots,w^s) \le (P',w^{1\prime},\ldots,w^{s\prime})$ in $\Pi_{n,s+1}$, then $A(P',w^{1\prime},\ldots,w^{s\prime}) \le A(P,w^1,\ldots,w^s)$ in the lexicographical ordering on atoms. 
\item[(2)] Any chain $\hat{0}\lessdot C_1\lessdot \cdots\lessdot C_d$ that contains an edge $x\lessdot y$ such that $A(x) \ne A(y)$, is not increasing.
\item[(3)]If $(P,w^1,\ldots,w^s) \lessdot (P',w^{1\prime},\ldots,w^{s\prime})$ is an edge such that $A(P,w^1,\ldots,w^s) \ne A(P,w^{s\prime},\ldots,w^{s\prime})$ and $\lambda((P,w^1,\ldots,w^s) \lessdot (P',w^{1\prime},\ldots,w^{s\prime})) =(k,i,j)$, then $w_{k}^i>w_k^{i \prime}=j$.
\item[(4)] If $(P,w^1,\ldots,w^s) \lessdot (P',w^{1\prime},\ldots,w^{s\prime})$ is an edge such that $A(P,w^1,\ldots,w^s) \ne A(P,w^{s\prime},\ldots,w^{s\prime})$ and $\lambda((P,w^1,\ldots,w^s) \lessdot (P',w^{1\prime},\ldots,w^{s\prime})) =(k,i,j)$, then $P'$ is obtained from $P$ by merging the part that contains $k$ with the part whose $i$th label contains $j$. 
\item[(5)] Suppose $I=[(P,w^1,\ldots,w^s),(P',w^{1\prime},\ldots, w^{s\prime})]$ is an interval in $\Pi_{n,s+1}$, and that $(k,i,j)$ is the first triple such that $w_k^i \ne w_k^{i \prime} =j$. Then every maximal chain in $I$ has a unique edge with label $(k,i,j)$, and every edge label in $I$ is greater than or equal to $(k,i,j)$.
\end{itemize}
\end{lemma}

\begin{proof}
Condition (1) follows immediately from the definition of $\lambda.$\\

In Condition (2), $\lambda(x\lessdot y) =(k,i,j)$ where $k \le n-1$ and $i \le s$. Also, $\lambda(\hat{0} \lessdot C_1) =(n-1,s+m,0)$ for some $m>0$. Therefore $(n-1,s+m,0)>(k,i,j)$, which implies that this chain is not increasing. \\

We will now prove condition (3). Suppose for a contradiction that $j>w_k^i$. Now, the restriction of $A(P,w^1,\ldots,w^s)$ (respectively $A(P,w^{1\prime},\ldots,w^{s\prime})$) to the subword $w_1^i,\ldots,w_n^i$ (respectively $w_1^{i\prime},\ldots,w_n^{i\prime}$), is equal to the atom $A(P,w^i)$ (respectively $A(P,w^{i\prime})$) in $\Pi_{n,2}$. Therefore, by Condition (1), we have that $j<w_k^i$.\\ 

We will now prove condition (4). By the definition of $A(P,w^1,\ldots, w^s)$, if $k$ is the $\alpha$th largest element in a part $L$, then $w_{k}^i$ is the $\alpha$th largest element in the $i$th label of $L$. The same is true of any element and any part, not just the element $k$ and the part $L$. Suppose that, for a contradiction, $j$ is contained in the $i$th label of $L$. Then, $j<w_k^i$, means that $j$ is less than the $\alpha$th largest element in the $i$th label of $L$. Then there is an integer $k-\epsilon$ ($\epsilon$ is a positive integer) which is less than $k$, and is contained in the part $L$, which has the property that $w_{k-\epsilon}^i=j$. Thus, $w_{k-\epsilon}^i \ne w_{k-\epsilon}^{i\prime}$, which contradicts that $(k,i,j)$ is the first triple such that $w_k^i \ne w_k^{i\prime}$. \\

We will now prove condition (5). Suppose for a contradiction that there is an edge with label $(\bar k, \bar i,\bar j)<(k,i,j)$ in $I$. Suppose also that $(\bar k,\bar i,\bar j)$ is the lowest edge label in $I$. Suppose that $(\bar k,\bar i) \ne (k,i)$, and consider a maximal chain $C$ in $I$ that contains an edge with label $(\bar k,\bar i,\bar j)$. By condition (3), and the fact that $(\bar k,\bar i,\bar j)$ is the lowest edge label in $I$, it must be true that higher edges in $C$ with a label of the form $(\bar k,\bar i,*)$ have lower values of $*$. Then we have a contradiction, since this implies that $(k,i,j)$ is not the lowest index such that $w_k^i \ne w_k^{i\prime} =j$. Suppose now that $(\bar k,\bar i) = (k,i)$. Consider a chain $C$ that contains an edge with the label $(\bar k,\bar i,\bar j)$. Again, by condition (3), and the fact that $(\bar k,\bar i,\bar j)$ is the lowest label in $I$, we know that higher edges in $C$ with a label of the form $(\bar k,\bar i,*)$ have lower values of $*$. Therefore, the highest edge with a label of the form $(\bar k,\bar i,*)$ in $C$ is the edge $(\bar k,\bar i,\bar j)$. Again, this implies the contradiction that $(k,i,j)$ is not the earliest index such that $w_k^i \ne w_k^{i\prime} = j$. Therefore, $(k,i,j)$ is the lowest edge label in $I$. By condition (4), this label occurs when the part that contains $k$ is merged with the part whose $i$th label contains $j$, and therefore this edge label occurs only once in every maximal chain. It is clear that an edge with label $(k,i,j)$ must occur once in each maximal chain.

\end{proof}

\begin{thm}
The edge labeling $\lambda$ of $\Pi_{n,s+1}$ is an EL-labeling.\label{kjgkjbkb}
\end{thm}

\begin{proof}
Consider intervals in $\Pi_{n,s+1}$ of length 1. Such intervals contain exactly one edge, so that the restriction of $\lambda$ to this interval is an EL-labeling. We assume by way of induction that the restriction of $\lambda$ to any interval of length $m$ for some $m<n$ is an EL-labeling. We will now show that $\lambda$ restricts to an EL-labeling on any interval of length $m+1$.\\

Suppose $I=[(P,w^1,\ldots,w^s),(P',w^{1\prime},\ldots, w^{s\prime})]$ is an interval of length $m+1$, and $A(P,w^1,\ldots,w^s) \ne A(P',w^{1\prime},\ldots, w^{s\prime})$. Suppose that $(k,i,j)$ is the lexicographically first element of $\mathbb{Z}\times \mathbb{Z} \times \mathbb{Z}$ such that $w_k^i \ne w_k^{i\prime}$, and $w_k^{i\prime} =j$. By Lemma \ref{htchft} part (5), every chain in $I$ has a unique edge with label $(k,i,j)$, and all edges have label greater than or equal to $(k,i,j).$ This implies that there is a unique atom $a$ in $I$ such that $\lambda((P,w^1,\ldots,w^s) \lessdot a)=(k,i,j)$, which, by Lemma \ref{htchft} part (4), is obtained from $(P,w^1,\ldots,w^s)$ by merging the part that contains $k$ with the part whose $i$th label contains $j$. By induction, the interval $[a, (P',w^{1 \prime},\ldots, w^{s \prime})]$ contains a unique increasing maximal chain which lexicographically precedes other chains. The concatenation of this chain with $(P,w^1,\ldots,w^s) \lessdot a$ gives a chain $C'$ which is the increasing because all edge labels in $[a, (P',w^{1 \prime},\ldots, w^{s \prime})]$ must be greater than $(k,i,j)$. By induction, $C'$ is the unique increasing maximal chain in $I$ which contains $(P,w^1,\ldots,w^s) \lessdot a$, and it lexicographically precedes other chains containing $(P,w^1,\ldots,w^s) \lessdot a$. Since all chains in $[(P,w^1,\ldots,w^s),(P',w^{1\prime},\ldots, w^{s\prime})]$ must contain an edge with label $(k,i,j)$, and chains that don't contain $(P,w^1,\ldots,w^s) \lessdot a$ have lowest edge label greater than $(k,i,j)$, $C'$ is the unique increasing maximal chain in $I$ which lexicographically precedes all other chains.\\

Suppose $[(P,w^1,\ldots,w^s),(P',w^{1\prime},\ldots, w^{s\prime})]$ is an interval of length $m+1$, and $A(P,w^1,\ldots,w^s) =A(P',w^{1\prime},\ldots, w^{s\prime})$. Recall that, by Proposition \ref{hgchc}, $$[(P,w^1,\ldots,w^s),(P',w^{1\prime},\ldots, w^{s\prime})] \cong [P,P'],$$ and denote this map by $\phi$. Observe also that $\lambda(x \lessdot y) = (n,\chi(\phi(x) \lessdot \phi(y)),0)$ for all $x \lessdot y \in [(P,w^1,\ldots,w^s),(P',w^{1\prime},\ldots, w^{s\prime})]$, where $\chi$ is the edge-labeling given in Theorem \ref{wathm}. Hence there is a unique increasing maximal chain that lexicographically precedes other chains in this interval.\\

Suppose $I=[\hat{0}, (P,w^1,\ldots,w^s)]$ is an interval of length $m+1$. Then, by the inductive hypothesis, the interval $[A(P,w^1,\ldots,w^s),(P,w^1\ldots,w^s)]$ contains a unique increasing maximal chain that lexicographically precedes other chains in this interval. The concatenation of this chain with $\hat{0} \lessdot A(P,w^1,\ldots,w^s)$, denoted $C'$, is increasing, since by Lemma \ref{htchft} part (1) all elements in the interval $[A(P,w^1,\ldots,w^s),(P,w^1\ldots,w^s)]$ have the same atom, and so all edge labels are of the form $(n,*,*)$. Therefore, by induction, $C'$ is the unique increasing chain which lexicographically precedes chains in $I$ that contain $\hat{0}\lessdot A(P,w^1,\ldots,w^s)$. Since $A(P,w^1,\ldots,w^s)$ is defined as the earliest atom that is less than $(P,w^1,\ldots,w^s)$, it is clear that $C'$ lexicographically precedes all maximal chains in $I$. Any chain in $I$ that does not contain $\hat{0}\lessdot A(P,w^1,\ldots,w^s)$ contains a descent, since they contain an edge $x \lessdot y$ in which $A(x) \ne A(y)$, so by Lemma \ref{htchft} part(2), they are not increasing.\\
\end{proof}

We provide a proof of the following Theorem, which is immediately implied by the main results of \cite{sa}:
\begin{thm}[Sagan \cite{sa}]
For any $n$ and $s$, the order complex $\Delta(\Pi_{n,s+1})$ is homotopy equivalent to a wedge of $(n-2)$-dimensional spheres.   
\end{thm}

\begin{proof}
By Theorem \ref{kjgkjbkb}, the poset $\Pi_{n,s+1}$ admits an EL-labeling, so that by Theorem \ref{kjvgm}, the order complex $\Delta(\Pi_{n,s+1})$ is shellable. By Proposition \ref{htgdyh}, this implies that $\Delta(\overline{\Pi}_{n,s+1})$ is shellable, and homotopy equivalent to a wedge of spheres. The spheres are of dimension $(n-2)$ since $\overline{\Pi}_{n,s+1}$ is pure, and its maximal chains contain $n-1$ elements.  
\end{proof}


\section{The number of spheres in $\Pi_{n,s+1}$}

Recall Theorem \ref{kjvgm}, which states that the number of spheres is equal to the number of decreasing maximal chains in $\Pi_{n,s+1}$. In this section, we characterise the decreasing maximal chains of $\Pi_{n,s+1}$ in order to count them. In the case that $s=1$, we show that the number of decreasing maximal chains is equal to the number of complete non-ambiguous trees with $n$ internal vertices, which are defined in \cite{abb}.

\begin{prop}
If $C = C_0 \lessdot C_1 \lessdot \cdots \lessdot C_n$ is a decreasing maximal chain in $\Pi_{n,s+1}$, then $A(C_i) \ne A(C_{i+1})$ for all $i \in [n-1]$. \label{gfdzxgx}
\end{prop}

\begin{proof}
Now $\lambda(C_0 \lessdot C_1) = (n-1,s+j,0)$ for some $j \ge 1$. By the definition of $\lambda$, if $A(C_i) = A(C_{i+1})$ for some $i \in [n-1]$, then $\lambda(C_i \lessdot C_{i+1}) =(n,*,*)$, which implies the contradiction that $\lambda(C)$ is not increasing. Therefore, must have that $A(C_i) \ne A(C_{i+1})$ for all $i \in [n-1]$.   
\end{proof}

\begin{example}
The chain 
\begin{gather*}C=\hat{0} \lessdot (\{1\}\{2\}\{3\}\{4\}\{5\},\{5\}\{4\}\{1\}\{3\}\{2\}) \lessdot (\{1\}\{2,4\}\{3\}\{5\},\{5\}\{34\}\{1\}\{2\})\\
\lessdot (\{1\}\{2,3,4\}\{5\},\{5\}\{1,3,4\}\{2\}) \lessdot (\{1,5\}\{1,3,4\},\{25\}\{134\})\\
\lessdot (\{1,2,3,4,5\},\{1,2,3,4,5\})
\end{gather*}
is maximal decreasing in the poset $\Pi_{5,2}$. The decreasing sequence of edge labels for this chain is $\lambda(C)=(4,5+116,0)(2,1,3)(2,1,1)(1,1,2)(1,1,1)$. This example also demonstrates the condition on decreasing chains given in Propositions \ref{mnhv} and \ref{yhthch}. 
\end{example}

\begin{prop}
Suppose that $(P,w^1,\ldots,w^s) \lessdot (P',w^{1 \prime},\ldots,w^{s \prime})$ is an edge in a decreasing maximal chain of $\Pi_{n,s+1}$, in which $\lambda((P,w^1,\ldots,w^s) \lessdot (P',w^{1\prime},\ldots,w^{s\prime}))=(k,i,j)$. Then $P'$ is of the from $\{1\}\{2\} \ldots \{k-1\}I_{k}\ldots I_{m}$, where $I_k$ is a non singleton set. The partition $P$ is of the form 
$\{1\}\{2\}\ldots\{k-1\}I_{\alpha}I_{k+1} \ldots I_{\beta} \ldots I_{m}$ where $I_{\alpha}$ and $I_{\beta}$ are a partition of $I_k$ (note that $I_{\beta}$ may be to the left of $I_{k+1}$ or to the right of $I_m$). For all $h<i$, the $h$th label of $I_{\alpha}$ contains the minimal element of the $h$th label of $I_k$, and the $i$th label of $I_{\beta}$ contains the minimal element of the $i$th label of $I_k$     
\label{mnhv}
\end{prop}

\begin{proof}
First, we will show that $P'$ is of the form $\{1\}\ldots\{k-1\}I_k\ldots I_m$, where $I_k$ is a non-singleton set.\\

Assume that $(P,w^1,\ldots,w^s)$ is an atom of $\Pi_{n,s+1}$, so that $P = \{1\} \ldots \{n\}$. Then $P' = \{1\}\ldots\{k-1\}\{kf\}\ldots\{n\}$, where $\{f\} \in \{k+1,\ldots,n\}$. This is true since $w_k^i \ne w_k^{i \prime}$ implies that $w_f^i \ne w_f^{i \prime}$, and so $f<k$ contradicts that $(k,i,j)$ is the first triple such that $w_k^i \ne w_k^{i \prime}$.\\

 Assume by induction that $P'$ is of this form if $P'$ has $m$ parts where $n-1 \le  m \le 2$. We will show that $P'$ has this form when it has $m+1$ parts. If $C_{n-m-1} \lessdot (P,w^1,\ldots,w^s) \lessdot (P',w^{1\prime},\ldots,w^{s\prime})$ is contained in a maximal decreasing chain, then $\lambda(C_{n-m-1} \lessdot (P,w^1,\ldots,w^s)) \ge (k,i,j)$, and by induction, $P=\{1\}\ldots \{k-1\}I_{k}......I_{m+1}$, where $I_{k}$ contains $k$ and is a non-singleton set exactly when $\lambda(C_{n-m-1} \lessdot (P,w^1,\ldots,w^s)) = (k,*,*)$. By Lemma \ref{htchft} Part (4), $P'$ is obtained from $P$ by merging the part that contains $k$ with the part whose $i$th label contains $j$. Now the part whose $i$th label contains $j$ cannot be any of $\{1\},\ldots,\{k-1\}$, as this would imply that $(k,i,j)$ is not the first triple such that $w_k^i \ne w_k^{i \prime}$. Hence $P'$ is of the form $\{1\}\ldots\{k-1\}I_k\ldots I_m$ where $I_k$ is a non-singleton set.  \\

By Lemma \ref{htchft} Part (4), $(P',w^{1 \prime},\ldots,w^{s \prime})$ is obtained from $(P,w^1,\ldots,w^s)$ by merging the part whose $i$th label contains $j$ (the set $I_{\beta}$) with the part that contains $k$ (the set $I_{\alpha})$. Since $k$ is the least integer in the set $I_{\alpha}$, it is only possible that $(k,i,j)$ is the first triple such that $w_k^i \ne w_k^{i\prime}=j$ if for every $h<i$, the minimal element of the $h$th label of $I_k$ is contained in the $h$th label of $I_{\alpha}$, and the minimal element of the $i$th label of $I_k$ (the number $j$) is contained in the $i$th label of $I_{\beta}$.

\end{proof}

\begin{example}
The chain $\hat{0}=C_0 \lessdot C_1 \lessdot \cdots \lessdot C_5$ in $\Pi_{5,3}$ is maximal decreasing.  
$$C_5=([5],[5],[5])$$
$$C_4=(\{1,4,5\}\{2,3\},\{2,3,5\}\{1,4\},\{3,4,5\}\{1,2\})$$
$$C_3 =(\{1\}\{2,3\}\{4,5\},\{2\}\{1,4\}\{3,5\},\{4\}\{1,2\}\{3,5\})$$
$$C_2 = (\{1\}\{2\}\{3\}\{4,5\},\{2\}\{4\}\{1\}\{3,5\},\{4\}\{1\}\{2\}\{3,5\})$$
$$C_1= (\{1\}\{2\}\{3\}\{4\}\{5\},\{2\}\{4\}\{1\}\{3\}\{5\},\{4\}\{1\}\{2\}\{5\}\{3\} )$$
$$C_0 = \hat{0}.$$
Then
\begin{gather*}\lambda(C_0 \lessdot C_1) = (4,2+m,0) \hbox{~for some $m$},~~~\lambda(C_1 \lessdot C_2) =(4,2,3),~~~\lambda(C_2 \lessdot C_3) = (2,1,1)\\
\lambda(C_3 \lessdot C_4) = (1,2,3),~~~\lambda(C_4 \lessdot C_5) = (1,1,1),\end{gather*}  
and 
\begin{gather*}A(C_1) = (24135)(41253),~~A(C_2) =(24135)(41235),~~~A(C_3) = (21435)(41235)\\ A(C_4) = (21435)(31245),~~~A(C_5) =(12345)(12345)\end{gather*} where we have added brackets into the expressions for the atoms for clarity. \label{chainex}
\end{example}

\begin{prop}
Suppose $C_d\lessdot\cdots \lessdot C_n$ is the upper portion of a maximal decreasing chain in $\Pi_{n,s+1}$. Suppose that $C_d=(P,w^1,\ldots, w^s)$, where $P = \{1\}\ldots\{k-1\}I_k\ldots I_{n-d+1}$ and $I_k$ is a non-singleton set. Suppose that $C_{d-1} \lessdot C_d$ is of the following form:
\begin{itemize}
\item If $\lambda(C_d \lessdot C_{d+1}) = (k,i,j)$ for some $i$ and $j$, then the partition in $C_{d-1}$ is of the form $\{1\}\ldots\{k-1\}I_{\alpha}I_{k+1}\ldots I_{\beta}\ldots I_{n-d+1}$, where $I_{\alpha}$ and $I_{\beta}$ are a partition of $I_{k}$ (it is possible that $I_{\beta}$ is to the left of $I_{k+1}$ or to the right of $I_{n-d+1}$) and $k \in I_{\alpha}$. Also, for some $i'\ge i$, the $i'$th label of $I_{\beta}$ contains $w_k^i$ (the minimal element of the $i$th label of $I_k$), and for all $h<i'$, the $h$th label of $I_{\alpha}$ contains $w_k^h$.
\item If $\lambda(C_d \lessdot C_{d+1}) = (k',i,j)$ for some $k'<k$, or if $d=n$, then $C_{d-1}$ is obtained by partitioning $I_k$ into two non-empty sets $I_{\alpha}$ and $I_{\beta}$ so that $k \in I_{\alpha}$, and for some index $i'$ the $i'$th label of $I_{\beta}$ contains $w_k^i$. 
\end{itemize}
then $C_{d-1} \lessdot \cdots \lessdot C_n$ is the upper portion of a maximal decreasing chain in $\Pi_{n,s+1}$. \label{yhthch}
\end{prop}

\begin{proof}
It is immediate from the definition of $\lambda$ that if $C_{d-1}$ is obtained using these methods, then the chain $C_{d-1} \lessdot \cdots \lessdot C_n$ is decreasing and $A(C_i) \ne A(C_j)$ when $i \ne j$. The element $C_{d-1}$ is of the same form as $C_d$, and so we can continue to extend this chain using this process, and any resulting chain will also have the property that it is decreasing. After performing $d-1$ iterations and concatenating the chain with $\hat{0}$, the resulting chain will be decreasing maximal. \\
\end{proof}

Let $\mathcal{B}_{n}^{s+1}$ denote the set of maximal decreasing chains in $\Pi_{n,s+1}$. For $n>1$, let $\mathcal{B}_n^{s+1,i}$ denote the set of maximal decreasing chains in $\Pi_{n,s+1}$ whose highest edge has label $(k,i,j)$ for some $k,j$. 

\begin{prop}
Any given $C  \in \mathcal{B}_n^{s+1,i}$ is defined uniquely by the following data. Here we index the $s+1$ partitions in any element of $\Pi_{n,s+1}$ by the integers $0$ to $s$:
\begin{itemize}
\item[(1)] The element $C_{n-1}$ in $C$. The element $C_{n-1}$ can be described as a partition into a left and right part, in which the left part is of size $\alpha$, of each set in $C_n = ([n],\ldots,[n])$. For all $h<i$, the $h$th element of $([n],\ldots,[n])$ is partitioned so that the left part contains $1$, and the $i$th element in $([n],\ldots, [n])$ is partitioned so that $1$ is contained in the right part.
\item[(2)] An element $C_L \in \mathcal{B}_\alpha^{s+1,i'}$ for some $2 \le \alpha \le n-1$ and some $i' \ge i$, or by an element $C_L \in \mathcal{B}_\alpha^{s+1}$ where $\alpha=1$. 
\item[(3)] An element $C_R \in \mathcal{B}_{n-\alpha}^{s+1}$.\label{hgchjv}  
\end{itemize}
\end{prop}

\begin{proof}
To prove this we first use Propositions \ref{gfdzxgx} \ref{mnhv} and \ref{yhthch} to characterise decreasing maximal chains in $\Pi_{n,s+1}$. Together, these propositions imply that any given decreasing maximal chain in $\Pi_{n,s+1}$ can be obtained by starting with the element $C_n = ([n],\ldots,[n])$ and repeatedly applying steps (1) or (2) of Proposition \ref{yhthch}. Also, by Proposition \ref{yhthch}, by repeatedly applying steps (1) or (2), one will obtain a maximal decreasing chain. Therefore, we can characterise maximal decreasing chains of $\Pi_{n,s+1}$ by considering the ways we can apply steps (1) or (2). We will first argue that (1)-(3) partially define elements of $\mathcal{B}_n^{s+1,i}$. We will then argue that (1)-(3) define these elements uniquely. \\

It is clear that point (1) partially describes elements in $\mathcal{B}_n^{s+1,i}$.\\

Suppose $C_{n-1} = (P,w^1,\ldots,w^s)$, where $P = I_1I_2$, and suppose that $|I_1| = \alpha$. Consider all of the parts and labels in elements in $C$ that are subsets of $I_1$ and its labels. Form a new chain $C_L \in \mathcal{B}_\alpha^{s+1}$ as follows. From each element $C_p \in C$ where $p \le n-1$ remove all parts that are subsets of $I_2$ and remove their corresponding labels. Next, relabel the elements in each part and label so that the $\alpha$th greatest element in $I_1$ is replaced with $\alpha$, and the $\alpha$th greatest element in the labels of $I_1$ are replaced with $\alpha$. Now, some of the elements in $\Pi_{\alpha,s+1}$ we obtained will be duplicates. If we remove the duplicates, then the set of elements form a maximal decreasing chain in $\mathcal{B}_{\alpha}^{s+1,i'}$ for some $i' \ge i$, or in $\mathcal{B}_1$ (see Example \ref{leftex}). This is the chain $C_L$ that partially defines $C$ in point (2). \\

Similar to the method used to obtain $C_L$, we obtain $C_R$ by removing parts that contain $I_1$ and their corresponding labels, and we then renumber the elements in each part and label (see Example \ref{leftex}). Thus, $C$ is also defined in part by an element $C_R \in \mathcal{B}_{n-\alpha}^{s+1}$, as described in point (3).\\

We have shown that points (1)-(3) are properties of any chain $C \in \mathcal{B}_n^{s+1,i}$. Now given the data of (1)-(3) for a chain $C$, we will argue that we can construct $C$ uniquely from this data. We prove this by induction. Assume that we are able to deduce the elements $C_n,C_{n-1},\ldots, C_d$ uniquely for some $d \le n-1$. Recall that $C_{n-1}$ has partition $I_1I_2$. Assume that $C_d = (P,w^1,\ldots, w^s)$, where $P=\{1\},\ldots,\{k-1\}I_k,\ldots,I_{}$, where $I_k$ is a non singleton set. Now the part $I_k$ is contained in either $C_L$ or $C_R$, and this can be deduced by considering whether the elements in $I_k$ are in $I_1$ or $I_2$ of $C_{n-1}$. We partition the set $I_k$ and its corresponding labels in the same manner that they are partitioned in the corresponding chain $C_L$ or $C_R$. This determines $C_{d-1}$ uniquely, and therefore, by induction, the entire chain $C$ can be determined uniquely.
\end{proof}

\begin{example}
Suppose $C$ is the maximal decreasing chain from Example \ref{chainex}. We rewrite it here for convenience. 
$$C_5=([5],[5],[5])$$
$$C_4=(\{1,4,5\}\{2,3\},\{2,3,5\}\{1,4\},\{3,4,5\}\{1,2\})$$
$$C_3 =(\{1\}\{2,3\}\{4,5\},\{2\}\{1,4\}\{3,5\},\{4\}\{1,2\}\{3,5\})$$
$$C_2 = (\{1\}\{2\}\{3\}\{4,5\},\{2\}\{4\}\{1\}\{3,5\},\{4\}\{1\}\{2\}\{3,5\})$$
$$C_1= (\{1\}\{2\}\{3\}\{4\}\{5\},\{2\}\{4\}\{1\}\{3\}\{5\},\{4\}\{1\}\{2\}\{5\}\{3\} )$$
$$C_0 = \hat{0}.$$

Here $I_1=\{1,4,5\}$ and $I_2 =\{2,3\}$. After removing subsets of $I_2$ and their corresponding labels, $C_3$, for example, becomes 
$$(\{1\}\{4,5\},\{2\}\{3,5\},\{4\}\{3,5\}),$$ and under the relabeling this becomes $$(\{1\}\{2,3\},\{1\}\{2,3\},\{2\}\{1,3\}).$$

Then $C_L$ is the chain $$\hat{0} \lessdot (\{1\}\{2\}\{3\},\{1\}\{2\}\{3\},\{2\}\{3\}\{1\})\lessdot (\{1\}\{2,3\},\{1\}\{2,3\},\{2\}\{1,3\})\lessdot ([3],[3],[3]),$$ and $C_R$ is the chain $$\hat{0} \lessdot (\{1\}\{2\},\{2\}\{1\},\{1\}\{2\}) \lessdot ([2],[2],[2]).$$\label{leftex}
\end{example}

Using Proposition \ref{hgchjv} we are able to count the number of maximal decreasing chains in $\Pi_{n,s+1}$. We do not find an exact solution for any $n$ and $s$, but we do find a recurrence relation which simplifies considerably when $s=1$. In the $s=1$ case, a solution to the recurrence is not known, but we show that the recurrence is the same as the recurrence found in \cite{abb} which counts the number of complete non-ambiguous trees. \\

\begin{thm}
For any $n \ge 2$, any $s\ge 1$, and any $1 \le i \le s$ the following recursion holds
\begin{equation}|\mathcal{B}_n^{s+1,i}| = \sum_{1 \le \alpha \le n-1} \sum_{i' \ge i}| \mathcal{B}_\alpha^{s+1,i'}||\mathcal{B}_{n-\alpha}^{s+1}|\binom{n-1}{\alpha-1}^{i}\binom{n-1}{\alpha}\binom{n}{\alpha}^{s-i},\label{recursion}\end{equation}

where, $|\mathcal{B}_{1}^{s+1}|=1$ for all $s$. When $\alpha=1$ in Equation \ref{recursion}, we let $\mathcal{B}_{\alpha}^{s+1,i'}=\mathcal{B}_{\alpha}^{s+1}$, and there is no summation over $i'$. \label{yhtgch}
\end{thm}

\begin{proof}
We have characterised decreasing maximal chains in Proposition \ref{hgchjv}.\\

The index $\alpha$ in Equation \ref{recursion} is the size of the left parts in point (1). If the left parts are of size $\alpha$, then there are $\binom{n-1}{\alpha-1}^{i}$ choices for the partitions the parts indexed from $0$ to $i-1$, since the left parts all contain $1$, leaving a choice of $\alpha-1$ from $n-1$ remaining elements in the left parts. There are $\binom{n-1}{\alpha}$ ways to choose elements for the left part of the $i$th label in $C_n=([n],\ldots,[n])$, and there are $\binom{n}{\alpha}^{s-i}$ ways to partition the remaining labels in $C_n=([n],\ldots,[n])$. These choices contribute the $\binom{n-1}{\alpha-1}^{i}\binom{n-1}{\alpha}\binom{n}{\alpha}^{s-i}$ terms in Equation \ref{recursion}.\\

 For conditions (2) and (3), the number of choices for $C_L$ and $C_R$ for any particular $\alpha$ and label $i' \ge i$ of the $C_L$, is given by $|\mathcal{B}_{\alpha}^{s+1,i'}||\mathcal{B}^{s+1}_{n-\alpha}|$. 

\end{proof}

We will now show that the set of maximal decreasing chains in $\Pi_{n,2}$ is in one to one correspondence with the set of complete non-ambiguous trees with $n$ leaves, which are defined and studied in \cite{abb}. Theorem \ref{yhtgch} shows that when $s=1$, the number of maximal decreasing chains in $\Pi_{n,2}$ satisfies the following recurrence: \\

\begin{equation}|\mathcal{B}_{n}^2| = \sum_{1 \le \alpha \le n-1}\binom{n-1}{\alpha-1}\binom{n-1}{\alpha}|\mathcal{B}_\alpha^2||\mathcal{B}_{n-\alpha}^2|, \label{jhvkkjb}\end{equation} 
with $\mathcal{B}_1^2=1$.\\

In \cite{abb}, the number of non-ambiguous trees with $n$ internal vertices is denoted by $b_n$ and is shown to satisfy the recurrence 

\begin{equation}b_{n+1} = \sum_{i+j=n}\binom{n+1}{i}\binom{n+1}{j}b_ib_j, \label{htyjhgc} \end{equation} for $n \ge 0$, with $b_0=1$. 

\begin{corol}
The order complex $\Delta(\Pi_{n,2})$ is homotopy equivalent to a wedge of $b_{n-1}$ $(n-2)$-spheres, where $b_n$ is the number of complete non-ambiguous trees with $n$ internal vertices.  
\end{corol}

\begin{proof}
By equations \ref{jhvkkjb} and \ref{htyjhgc} it follows that $b_{n-1} = |\mathcal{B}_{n}^2|$ for all $n \ge 1$.
\end{proof}

\end{document}